\documentclass[12pt]{amsart}
\usepackage{amsmath}
\usepackage{amsthm}
\usepackage{amscd}
\usepackage{setspace}
\usepackage{enumerate}

\newcommand{\R}{{\mathbb R}}

\newcommand{\C}{{\mathbb C}}

\newtheorem{theorem}{Theorem}[section] 
\newtheorem{lemma}{Lemma}[section]            
\newtheorem{corollary}{Corollary}      
\newtheorem{proposition}[theorem]{Proposition}
\theoremstyle{definition}
\newtheorem*{acknowledgement}{Acknowledgement}

\newtheorem{definition}{Definition}
\numberwithin{definition}{section}

\theoremstyle{remark}
\newtheorem{remark}{Remark}

\begin{document}

\title[ Spectral strong multiplicity one]{On a spectral analogue of the 
 strong \\ multiplicity one theorem }

\date{}

\author{Chandrasheel Bhagwat}
\author{C. S. Rajan}
\address{Tata Institute of Fundamental Research\\
Homi Bhabha Road\\
Mumbai 400005, India.}

\email{chandra@math.tifr.res.in, rajan@math.tifr.res.in}

\begin{abstract}  We prove spectral analogues of the classical strong
multiplicity one theorem  for newforms. Let  $\Gamma_1$ and $\Gamma_2$
be uniform lattices in a semisimple group $G$. Suppose all but finitely
many irreducible unitary representations (resp. spherical)  of $G$
occur with equal  multiplicities in $L^{2}(\Gamma_1 \backslash G)$ and
$L^{2}(\Gamma_2 \backslash G)$. Then  $L^{2}(\Gamma_1 \backslash G)
\cong L^{2}(\Gamma_2 \backslash G)$ as  $G$ - modules (resp. the
spherical spectra of $L^{2}(\Gamma_1 \backslash G)$ and
$L^{2}(\Gamma_2 \backslash G)$ are equal).
\end{abstract}

\maketitle

\section{\bf{Introduction}}\label{intro}

The beginnings of the analogy between the spectrum and arithmetic of
Riemannian locally symmetric spaces can be attributed to Maass, who
defined non-analytic modular forms as eigenfunctions of the Laplacian
satisfying suitable modularity and growth conditions.  From the
viewpoint of Gelfand, the theory of Maass forms can be re-interpreted
in terms of the representation theory of $PSL(2,\R)$  on $L^{2}(\Gamma
\backslash PSL(2,\R))$ for a lattice $\Gamma$ in
$PSL(2,\R)$. Subsequently, the analogy between the spectrum  and
arithmetic has been extended by the work of A. Selberg, P. Sarnak,
M. F. Vigneras  and T. Sunada amongst others.

 In this paper, our aim is to establish an analogue in the spectral
context of the classical strong multiplicity one theorem for cusp
forms. Suppose $f$ and $g$ are newforms for some Hecke congruence
subgroup $\Gamma_{0}(N)$ such that the eigenvalues of the Hecke
operator at a prime $p$ are equal for all but finitely many primes
$p$. Then the strong multiplicity one theorem of Atkin and Lehner
states that $f$ and $g$ are equal (\textit{cf}. ~\cite[p.125]{La}).

Now, let $G$ be a semisimple Lie group and $\Gamma$ be a uniform
lattice (a discrete cocompact subgroup) in $G$. Let $R_{\Gamma}$ be
the right regular representation of $G$ on $L^2(\Gamma \backslash G)$:

$$R_{\Gamma}(g)(\phi)(y) = \phi(yg)\quad \forall\  g,~y \in G\ \text{and}\ \phi \in\ L^2(\Gamma \backslash G) $$ 

This defines a unitary representation of $G$. It is known
(\cite[p.23]{GGP}) that $R_{\Gamma}$ decomposes discretely as a direct
sum of irreducible unitary representations of $G$ occurring with
finite multiplicities. Let $\widehat{G}$ be the set of equivalence
classes of irreducible unitary representations of $G$. For $\pi \in
\widehat{G}$,  let $m(\pi, \Gamma )$\ be the multiplicity of $\pi$ in
$R_{\Gamma}$. 

\begin{definition}Let $\Gamma_1$ and $\Gamma_2$ be uniform lattices in
$G$.  The lattices $\Gamma_1$ and $\Gamma_2$ are said to be
representation equivalent in $G$ if 
$$L^{2}(\Gamma_1 \backslash G) \cong L^{2}(\Gamma_2 \backslash G)$$
 as representations of $G$, i.e. for every $\pi\ \in\ \widehat{G}$, 
\[m(\pi, \Gamma_1 ) = m(\pi, \Gamma_2)\]
\end{definition}

In this article we prove the following result :

\begin{theorem}\label{thm1} Let $\Gamma_1 $ and $\Gamma_2 $  be
uniform lattices in a semisimple Lie group $G$.  Suppose all but
finitely many irreducible unitary representations of $G$ occur with
equal multiplicities in  $L^{2}(\Gamma_1 \backslash G)$ and $L^{2}
(\Gamma_2 \backslash G)$. Then the lattices $\Gamma_1$ and $\Gamma_2$
are representation equivalent in $G$.
\end{theorem}

The proof of Theorem ~\ref{thm1} uses the Selberg trace formula and
fundamental results of Harish Chandra on the character distributions
of irreducible unitary  representations of $G$; in particular, we make
crucial use of a deep and difficult result  of Harish Chandra that the
character distribution of an irreducible unitary representation of $G$
is given by a locally  integrable function on $G$.  
\par We now consider an analogue of Theorem ~\ref{thm1} for the
spherical spectrum of uniform lattices.  Let $K$ be a maximal compact
subgroup of $G$. An irreducible unitary  representation $\pi$ of $G$
is said to be spherical if there exists a non-zero vector $v \in \pi$
such that
\[ \pi(k) v = v\quad \forall \ k \in K.\] The spherical spectrum
$\widehat{G}_s$ of   $G$ is the subset of $\widehat{G}$ consisting of
equivalence classes of irreducible unitary spherical  representations
of $G$.

\begin{theorem}\label{prin} Let $G$ be a connected, semisimple Lie
group. Suppose $\Gamma_1$, $\Gamma_2$ are uniform  torsion-free
lattices in $G$ such that 
$$ m(\pi, \Gamma_1 ) =  m(\pi, \Gamma_2 )$$ 
for all but finitely many representations $\pi$ in $\widehat{G}_s$. 
Then $$ m(\pi, \Gamma_1 ) =  m(\pi, \Gamma_2 )$$ for all
representations $\pi$ in $\widehat{G}_s$. 
\end{theorem} The proof of this theorem follows the broad outline of
Theorem \ref{thm1}, but requires a more delicate control of $K\times
K$-saturation of a conjugacy class of an element of $\Gamma$ (see
Proposition \ref{B}). This is achieved by looking at the behaviour of
the conjugacy classes of elements of $\Gamma$ in a neighbourhood of
identity.  

We now relate the spherical spectrum with the spectrum of
$G$-invariant differential operators on the associated symmetric space
$X = G/K$. For a torsion-free uniform lattice $\Gamma$ in $G$, let
$X_{\Gamma} = \Gamma \backslash G / K$ be the associated compact
Riemannian locally symmetric space.  The space of smooth functions on
$X_{\Gamma}$ can be considered as  the space of smooth functions on
$X$ invariant under the action of $\Gamma$. Let $D(G / K)$ be the
algebra of $G$-invariant differential operators on $X$.  For a
character $\lambda$ of $D(G/K)$ (i.e. an algebra homomorphism of
$D(G/K)$ into $\C$), consider the eigenspace of $\lambda$, 
\begin{equation} \label{V} V(\lambda,\Gamma) = \left\{f \in
C^{\infty}(X_{\Gamma})\ :\ D(f) = \lambda(D)(f)\quad \forall\ D \in
D(G / K) \right\}. 
\end{equation} It is known that the space $V(\lambda,\Gamma)$ is of
finite dimension (see Section ~\ref{prooflaplacian}).

\begin{definition} Let $\Gamma_1$ and $\Gamma_2$ be torsion-free
uniform lattices in $G$.  The locally symmetric spaces $X_{\Gamma_1} =
\Gamma_1 \backslash G / K$ and $X_{\Gamma_2} = \Gamma_2 \backslash G /
K$  are said to be \textit{compatibly isospectral} if 
$$\text{dim} ( V (\lambda, \Gamma_1) ) = \text{dim}
\ ( V (\lambda, \Gamma_2) )$$ for every character $\lambda$
of $D(G / K)$.
\end{definition}

\begin{remark} From the generalized Sunada criterion proved by Berard
\cite[p.566]{Be} and DeTurck - Gordon \cite{DG}, it can be seen that
if two uniform lattices in $G$ are representation equivalent, then
the associated compact locally symmetric Riemannian spaces
$X_{\Gamma_1}$ and  $X_{\Gamma_2}$ are compatibly isospectral. 
\end{remark}

We prove the following result in Section ~\ref{prinlapla} :

\begin{theorem}\label{laplacian}  Let $G$ be a connected, semisimple
Lie group. Suppose $\Gamma_1$, $\Gamma_2$ are uniform  torsion-free
lattices in $G$. Suppose 
$$\text{dim}\ ( V (\lambda, \Gamma_1 )) = \text{dim}
 ( V (\lambda, \Gamma_2 ) )$$
for all but finitely many characters $\lambda$, then  $X_{\Gamma_1}$
and $X_{\Gamma_2} $  are compatibly isospectral.
\end{theorem}

If $X$ is of rank one, the algebra $D(G / K)$ is the polynomial
algebra in the  Laplace-Beltrami operator $\Delta$ on $G/K$ (see
~\cite[p.397]{He}). Hence the eigenvalues of $\Delta$  determine the
characters of $D(G / K)$. Consequently we get :

\begin{corollary}\label{lapla} Let $X_1$ and $X_2$ be two locally
symmetric Riemannian spaces of rank one and $\Delta_1$, $\Delta_2$ be
the Laplace-Beltrami operators acting on the space of smooth functions
on $X_1$ and $X_2$ respectively. If all but finitely many eigenvalues
occur with equal multiplicities in the spectra of $\Delta_1$ and
$\Delta_2$, then the spaces are isospectral with respect to the
Laplace-Beltrami operators. 
\end{corollary}

\begin{remark} Using an analytic version of the Selberg Trace formula,
J. Elstrodt, F. Grunewald, and J. Mennicke (on a suggestion of
M. F. Vigneras) proved Corollary ~\ref{lapla} for $G = PSL(2,\R)$ and
$G = PSL(2,\C)$ (\cite[Theorem 3.3, p.203]{EGM}).
\end{remark}

\begin{remark} When $G=PSL(2,\R)$, it can be seen that the spherical
spectrum of determines the full spectrum $L^2(\Gamma\backslash G)$
(\cite{Pe}). One can raise the question whether such a result will be
true in general. This fits in with the conjectures linking spectrum
and arithmetic in the context of automorphic forms  (see
\cite[Conjecture 3]{Ra}). 
\end{remark}

\begin{acknowledgement} We thank S. Kudla for raising the question of
proving an analogue of the strong multiplicity one theorem  in the
spectral case during the conference on `Modular forms' held at
Schiermonikoog, Netherlands in October 2006.  The second author thanks
the organizers of the conference for the invitation and warm
hospitality.\\
\end{acknowledgement}

\section{Preliminaries} \label{Prelim}

\subsection{Representations of semisimple groups  }\label{repgrp} We
recall some facts about representations of semisimple groups.  Let $G$
be a semisimple Lie group with a Haar measure $\mu$.    Let $\pi$ be a
unitary representation of $G$ on a Hilbert space $V$.  For a compactly
supported smooth function $f$ on $G$, define the convolution operator
$\pi(f)$ on $V$ as follows :

 \[\pi(f)(v) = \int\limits_G{f(g)\ \pi(g)v\ d\mu(g)}\]
 
This defines a bounded linear operator on $V$.  We recall the
following result from ~\cite[Theorem 10.2; p.334]{Kn} :

\begin{proposition}\label{trace} Let $G$ be a semisimple group and
$\pi$ be an irreducible unitary representation of $G$.  Then the
convolution operator $\pi(f)$ is of trace class for every compactly
supported smooth function $f$ on $G$.
\end{proposition}

Let $\pi$ be an irreducible unitary representation of $G$. Let
$C_{c}^{\infty}(G) $ be the space of  compactly supported smooth
functions on $G$.  Define the character distribution $\chi_{\pi}$ by,
$$\chi_{\pi}(f) = \text{trace}~(\pi(f)) \quad \forall f \in C_{c}^{\infty}(G). $$

\subsection{Some results of Harish Chandra}\label{HC} 
We recall some results of Harish Chandra on the characters of
irreducible unitary representations of $G$.

\begin{theorem}\label{LinInd} ~\cite[Theorem 10.6; p.336]{Kn}  Let
$\left\{\pi_i\right\}$ be a finite collection of mutually
inequivalent irreducible unitary representations of $G$.  Then their
characters $\left\{\chi_{\pi_i}\right\}$ are linearly  independent
distributions on $C_{c}^{\infty}(G)$. 
\end{theorem}

Let $L^{1}_{loc}(G)$ be the space of all complex valued  measurable
functions $f$ on $G$ such that 
$$ \int\limits_{C} |f(g)|\ d\mu (g)\ < \infty\ \ \  \text{for all compact subset}\ C\ \text{of}\ G.$$

The following deep result of Harish Chandra will be crucially used in
the proof of Theorem ~\ref{thm1}.

\begin{theorem}\label{locint} ~\cite[Theorem 10.25; p.356]{Kn} Let \
$\pi$ be an irreducible unitary representation of $G$. The
distribution character  $ \chi_{\pi}$ is given by a locally integrable
function $h$ on $G$. i.e. there exists $ h \in L^{1}_{loc}(G)$ such
that 
 $$\chi_{\pi}(f) = \int \limits_{G}\ f(g)\ h(g)\ d\mu(g)
\quad \forall f \in C^{\infty}_c(G).$$
 \end{theorem}

\subsection{Selberg trace formula for compact quotient}\label{STF} We
recall the Selberg trace formula for compact quotient. (For details,
see Wallach ~\cite[p.171-172]{Wa}).  Let $\mu'$ be the normalized
$G$-invariant measure on the quotient space $\Gamma \backslash G$.
For a compactly supported smooth function $f$ on $G$, the convolution
operator $R_{\Gamma}(f)$ on $L^2(\Gamma \backslash G)$ is given by :

\[
\begin{split} R_{\Gamma}(f)(\phi)(y)\ \ \ = \int \limits_{G}f(x)\
\phi(yx)\ d\mu(x) \ \ \ \ \ \ \forall\ \phi \in L^2(\Gamma \backslash
G)\  \text{and}\ y \in G.
\end{split}
\]

\[
\begin{split} = \int \limits_{\Gamma \backslash G}\left[
\sum\limits_{\gamma\ \in\ \Gamma}f(y^{-1} \gamma x) \right]\ \phi(x)\
d\mu'(x).
\end{split}
\]

Since $f$ is a smooth and compactly supported function on $G$ and
$\Gamma$ is  uniform lattice, the sum $K_{f}(y,x) =
\sum\limits_{\gamma\ \in\ \Gamma}f(y^{-1} \gamma x)$ is a finite sum,
and hence it follows that the operator $R_{\Gamma}(f)$ is of
Hilbert-Schmidt class. The trace of $R_{\Gamma}(f)$ is defined and it
is given by integrating the kernel function $K_{f}(y,x)$,

\[
\begin{split} \text{tr}(R_{\Gamma}(f)) = \int \limits_{\Gamma
\backslash G} \left[\sum_{\gamma\ \in\ \Gamma} f(x^{-1}\gamma
x)\right] d\mu'(x)
\end{split}
\]

Let $[\gamma]_{G}$ (resp. $[\gamma]_{\Gamma}$) be the conjugacy class
of $\gamma$ in $G$ (resp. in $\Gamma$). Let $\left[\Gamma\right]$
(resp.  $\left[\Gamma \right]_G$) be the set of conjugacy classes in
$\Gamma$ (resp. the $G$-conjugacy classes of elements in $\Gamma$).
For $\gamma \in \Gamma$, let $G_{\gamma}$ be the centralizer of
$\gamma$ in $G$. Put $\Gamma_{\gamma} = \Gamma \cap G_{\gamma}$. It
can be seen that $\Gamma_{\gamma}$  is a lattice in $G_{\gamma}$ and
the quotient $\Gamma_{\gamma} \backslash G_{\gamma}$ is compact.
Since $G_{\gamma}$ is unimodular, there exists a $G$-invariant measure
on $G_{\gamma} \backslash G$, denoted by $d_{\gamma}x$. After
normalizing the measures on $G_{\gamma}$ and $G_{\gamma} \backslash G$
appropriately and rearranging the terms on the right  hand side of
above equation, we get :

\begin{equation}\label{1} \text{tr}(R_{\Gamma}(f)) \ \ \ \ \ \ \ =
\sum \limits_{[\gamma]\ \in \ \left[\Gamma \right]} \text{vol}
(\Gamma_{\gamma} \backslash G_{\gamma}) \ \int\limits_{G_{\gamma}
\backslash G} f(x^{-1} \gamma x)\ d_{\gamma}x
\end{equation}

\begin{equation*}\label{geom} = \sum \limits_{[\gamma]\ \in\
[\Gamma]_G}a(\gamma,\Gamma)\ O_{\gamma}(f)
\end{equation*} where $O_{\gamma}(f)$ is the orbital integral of $f$
at $\gamma$ defined  by,  
\[ O_{\gamma}(f) = \int\limits_{G_{\gamma} \backslash G} f(x^{-1}
\gamma x)\ d_{\gamma}x.\] Here \[ a(\gamma, \Gamma) = \sum
\limits_{[\gamma']_\Gamma \ \subseteq\ [\gamma]_G}\ \text{vol} \
(\Gamma_{\gamma'} \backslash G_{\gamma'}). \]  If $\gamma$ is not
conjugate to an element in  $\Gamma$,  we define $a(\gamma,
\Gamma)=0$.  On the other hand, the trace of $R_{\Gamma}(f)$ on the
spectral side can be written as an absolutely convergent series as,
\begin{equation}\label{rep} \text{tr}(R_{\Gamma}(f)) = \sum
\limits_{\pi\ \in \ \widehat{G}} m(\pi,\Gamma)\chi_{\pi}(f)
\end{equation}
 
Hence from (\ref{1}) and (\ref{rep}), we obtain the Selberg trace
formula: 

\begin{equation}\label{STF1} \sum \limits_{\pi\ \in \ \widehat{G}}
m(\pi,\Gamma) \chi_{\pi}(f) = \sum \limits_{[\gamma]\ \in\
[\Gamma]_G}a(\gamma,\Gamma)\ O_{\gamma}(f). 
\end{equation}

\section{\bf{Proof of Theorem ~\ref{thm1}}}
\subsection{Some preliminary lemmas} We first recall some known
results about the geometry of conjugacy classes in $G$.

\begin{lemma}\label{conj} Let $\Gamma$ be a uniform lattice in
$G$. Let $\gamma \in \Gamma$.  Then the $G$-conjugacy class
$[\gamma]_G$ is a closed subset of measure zero in $G$. 
\end{lemma}

\begin{proof} 

Let $\left\{g_{n}^{-1} \gamma g_{n}\right\}_{n=1}^{\infty}$ be a
sequence of points in  $[\gamma]_G$ which converges to $h$ in
$G$. Since $\Gamma \backslash G$ is compact,  there exists a
relatively compact set $D$ of $G$ such that $G = \Gamma D$. Write  $
g_{n} = \gamma_{n}\ d_{n}$ where $\gamma_{n}\ \in\ \Gamma$ and $d_{n}\
\in\ D$. Hence,
\[ g_{n}^{-1}\ \gamma\ g_{n} = d_{n}^{-1}\ \gamma_{n}^{-1}\ \gamma\
\gamma_{n}\ d_{n}.\] Since $D$ is relatively compact, there is a
convergent subsequence of  $\left\{ d_{n} \right\}_{n=1}^{\infty}$,
which converges to some element $d$ of $G$.  Hence we get,  
\[\lim \limits_{n \rightarrow \infty} \gamma_{n}^{-1}\ \gamma\
\gamma_{n} = d^{-1}hd.\] Since $\Gamma$ is discrete, for large $n$,
\[\gamma_{n}^{-1}\ \gamma\ \gamma_{n} = d^{-1}hd.\] Hence $h\ \in\
[\gamma]_G$. Thus $[\gamma]_G$ is closed in $G$.

The conjugacy class $[\gamma]_G$ is homeomorphic to the homogeneous
space $G_{\gamma}\backslash G$.  Hence there exists a natural
structure of a smooth manifold on it such that it is a submanifold of
$G$.  Since $G_{\gamma}$ is non trivial (it contains a Cartan
subgroup of $G$), it is of lower dimension than $G$ and hence of
measure zero with respect to the Haar measure $\mu$ on $G$.

\end{proof}

\begin{lemma}\label{berard}  Let $\Omega$ be a relatively compact
subset of $G$. Then the set 
\[ A_{\Omega}  =  \left\{\ [\gamma]_G\ :\  \gamma \in \Gamma\
\text{and}\  [\gamma]_G\ \cap\ \Omega\ \neq\ \emptyset\ \right\}\] is
finite.
\end{lemma}  

\begin{proof} Let $x \in G$ be such that $x^{-1} \gamma x \in \Omega$.
As in Lemma ~\ref{conj}, write $x = \gamma_1. \delta$ where $\gamma_1
\in \Gamma$ and $\delta \in D$.  Hence $\gamma_1^{-1} \gamma \gamma_1
\in D \Omega D^{-1}$ which is relatively compact in $G$. Hence
$\gamma_1^{-1} \gamma \gamma_1  \in D \Omega D^{-1}\ \cap\ \Gamma$
which is a finite set.
\end{proof}

\begin{corollary}\label{E} Let $E$ be the union of the  conjugacy
classes $\ [\gamma]_G\ \text{such that}\ \gamma \in  {\Gamma_1} \cup
{\Gamma_2}$.   Then $E$ is a closed subset of measure zero in $G$.
\end{corollary}

\begin{proof} By using above two lemmas, it follows that $E \cap C$ is
finite for every compact subset $C \subseteq G$. Hence $E$ is closed
in $G$. It is of measure  zero since it is a countable union of sets
of measure zero.
\end{proof}

\subsection{Proof of Theorem ~\ref{thm1}} For $\pi \in \widehat{G}$,
let $t_\pi = m(\pi,\Gamma_1) - m(\pi,\Gamma_2)$.  Let $f \in
C^{\infty}_{c}(G)$. Since the series in equation (\ref{STF1})
converges absolutely, by comparing equation (\ref{STF1}) for
$\Gamma_1$ and $\Gamma_2$, we obtain:

\begin{equation*}\label{trf1} \sum \limits_{\pi\ \in \ \widehat{G}}
t_\pi\ \chi_{\pi}(f) = \sum \limits_{\substack{[\gamma]\  \in \
[\Gamma_1]_G\ \cup\ [\Gamma_2]_G}} (a(\gamma,\Gamma_1) -
a(\gamma,\Gamma_2)) \ O_{\gamma}(f). 
\end{equation*} By hypothesis, $t_\pi = 0$ \ for all but finitely many
$\pi \in \hat{G}$.   Hence there exists a finite subset $S$ of
$\widehat{G}$ such that, 

\begin{equation}\label{trf3}
\begin{split} \sum \limits_{\pi\ \in\ S} \ t_{\pi}\chi_{\pi}(f) =
\sum \limits_{\substack{[\gamma]\ \in \ [\Gamma_1]_G\ \cup\
[\Gamma_2]_G}}  (a(\gamma,\Gamma_1) - a(\gamma,\Gamma_2))\
O_{\gamma}(f). 
\end{split}
\end{equation} Since $S$ is a finite set, by Harish Chandra's Theorem
~\ref{locint}, there exists a function $\phi \in L^{1}_{loc}(G)$ such
that 
 
\begin{equation} \label{loc} \sum \limits_{\pi\ \in\ S}
t_{\pi}. \chi_{\pi}(f) = \int \limits_{G} f(g)\ \phi(g)\ d\mu(g) \quad
\forall\ f \in C^{\infty}_c(G). 
\end{equation} Let $E$ be as in Corollary \ref{E} above. Let $g \in G$
be any point outside $E$.  Since $E$ is closed in $G$, there exists a
relatively compact neighborhood $U$ of $g$  such that $U \cap E =
\emptyset $. Hence, if  $ f\ \in\ C^{\infty}_c(G) $  is supported on
$U$, we have 
\[O_{\gamma}(f) = 0 \quad \quad   \forall\ \gamma \in \Gamma_1 \cup
\Gamma_2.\]  Hence from equations  (\ref{trf3}) and (\ref{loc}) above,
we get : 
\[\int \limits_{G} f(g)\ \phi(g)\ d\mu(g) = 0,\] for all smooth
compactly supported functions $f$ supported in $U$.
 
But this means that $ \phi(g)$ is essentially $0$ on $U$.  Since $U$
was a  neighborhood  of an arbitrary point $g$ outside $E$, and $E$ is
a closed subset of measure zero, we conclude that $ \phi(g)$ is
essentially $0$ on $G$. By equation (\ref{loc}) above : 
\[\sum \limits_{\pi\ \in\ S} t_{\pi} \chi_{\pi}(f)\ =\ 0 \quad
\forall\ f\ \in C^{\infty}_c(G).\] From the linear independence of
characters (Theorem ~\ref{LinInd}),  we get that $ t_{\pi} = 0 $ for
any $\pi \in S$. Hence,
\[ m(\pi, \Gamma_{1}) = m(\pi, \Gamma_{2}) \quad \forall\ \pi \in
\widehat{G}.\] i.e., the lattices $\Gamma_1$ and $\Gamma_2$ are
representation  equivalent in $G$.

\section{Proof of Theorem ~\ref{prin}}\label{prinlapla} In this
section we give a proof of Theorem \ref{prin}, following the broad
outline of the proof of Theorem \ref{thm1}. Since the analogue of
Corollary \ref{E} does not seem available to us, we need to establish
a more delicate proposition concerning the $K\times K$-saturation
$KC_{\gamma}K$ of the conjugacy class of elements $\gamma\in
\Gamma$. Corresponding to use of Harish Chandra's theorem on the local
integrability of the character of an irreducible unitary
representation of $G$, we instead use the analyticity of the spherical
functions on $G$. 

\begin{definition}\cite[p.399]{GV} A complex valued function $\phi$ on
$G$ is called a spherical function if
\begin{enumerate}
  \item $\phi(e) = 1$.
  \item $\phi(k_{1} x k_{2}) = \phi(x)\quad  \forall\ k_{1},\ k_{2}
\in K\ \text{and}\ x \in G$.
  \item $\phi$ is a common eigenfunction for all $D$ in the space
$D(G/K)$ of $G$-invariant differential operators on $G / K$ with
eigenvalue $\lambda(D)$:
 \[ D\phi = \lambda(D)\phi \ \ \forall\ D \in D(G/K).\]
  \end{enumerate}
 \end{definition}  The map $D \rightarrow \lambda(D)$ defines a
algebra homomorphism   of $D(G/K)$ into $\C$. Denote by
$C^{\infty}_{c}(G//K)$  the space of smooth and compactly supported
bi-$K$-invariant functions on $G$.
   
Let $\pi$ be a spherical unitary representation of $G$. The space
$\pi^{K}$ of $K$-fixed vectors is one dimensional
(\textit{cf}. Helgason ~\cite[p.416]{He}).  Let $\phi_{\pi}$ be the
associated elementary spherical function defined by    
\[ \phi_{\pi}(x) = \left\langle\ \pi(x)\ e_{\pi}, e_{\pi}\
\right\rangle, \] where $e_{\pi}$ is a $K$-fixed vector of the
representation space of $\pi$ such that  $\|e_{\pi}\| = 1$. 

We have the following proposition: 
 \begin{proposition}\label{warn} Let $\pi$ be an irreducible unitary
spherical representation of $G$. Then the following hold:
\begin{enumerate}[(i)]
\item The associated elementary spherical functions $\phi_{\pi}$ are
analytic on $G$. 
\item The relationship of the elementary spherical function
$\phi_{\pi}$ to character $\chi_{\pi}$ is given by the following
equation:
\begin{equation*} \chi_{\pi}(f) = \int \limits_{G} f(g)\
\phi_{\pi}(g)\ dg \quad f\in C^{\infty}_{c}(G//K).
\end{equation*}
\item   Let $\left\{\pi_{j} : 1 \leq j \leq k \right\} $ be a finite
collection of  mutually inequivalent irreducible spherical
representations of $G$.  The associated elementary spherical functions
$\left\{\phi_{\pi_{j}} \ : 1 \leq j \leq k \right\} $ are linearly
independent. 
\end{enumerate}
\end{proposition} 
\begin{proof}
\begin{enumerate}[(i)]
\item Since the algebra $D(G/K)$ contains the Laplace-Beltrami
operator which is an elliptic,  essentially self adjoint differential
operator, it follows that the elementary spherical functions
$\phi_{\pi}$ are analytic on $G$.   
\item Let $V$ be the space underlying the representation $\pi$. Given
$f \in$ $ C^{\infty}_{c}(G//K)$, the image $\pi(f)(V)$ of the
convolution operator $\pi(f)$  lands in the space $V^K$ of
$K$-invariants. Hence the trace is given by,
\[ \chi_{\pi}(f) = \text{trace}~(\pi(f))   =\left\langle \pi(f)
e_{\pi}, e_{\pi} \right\rangle  = \int \limits_{G} f(g)\
\phi_{\pi}(g)\ dg.
\]
\item The function $\phi_{\pi_{j}}$ is an eigenvector for the
character $\lambda_{\pi_{j}}$.  Since the representations $\pi_{j}$
are mutually inequivalent, the homomorphisms $\lambda_{\pi_{j}}$ are
distinct and hence the corresponding eigenvectors are linearly
independent.
\end{enumerate}

\end{proof}
   
Now we turn to the geometric aspects of the Selberg trace formula. Let
$G//K$ denote the collection of orbits under the action of $K\times K$
acting on $G$ by, 
\begin{equation}\label{biKaction} (k, l)g=k^{-1}gl \quad \quad k, l\in
K, ~g\in G.
\end{equation}

\begin{lemma}\label{sph}   The space $C^{\infty}_{c}(G//K)$ consisting
of bi-$K$-invariant compactly  supported smooth functions on $G$
separate points on $G//K$.
   \end{lemma}
\begin{proof} The orbits of $K\times K$ being compact are closed
subsets of $G$. Given two orbits $KxK, ~KyK$ choose  a compactly
supported, smooth function which is positive on $KxK$ and vanishes on
$KyK$. Then, 
\[ F(g) =\int_{K\times K} f(kgl)dk dl,\] is a bi-$K$-invariant
compactly supported smooth function on $G$ which separates the two
orbits.  
   \end{proof}

\begin{lemma}\label{gamma} Let $\Gamma$ be a torsion-free uniform
lattice in $G$. For a non-trivial element $\gamma \in \Gamma$,   the
conjugacy class $C_{\gamma}$ is disjoint from $K$.  
\end{lemma}

\begin{proof}  The group  $x^{-1}\gamma x \cap\ K$ is discrete and
contained in the compact group $K$, hence finite.  Consequently,
$x^{-1}\gamma x$ is of finite order in $G$ and hence $\gamma$ is of
finite order. Since $\Gamma$ is torsion-free,  $\gamma$ is the
identity element of $G$. 
\end{proof}

\begin{lemma} \label{e} If $\gamma \neq e$, then $e \notin
KC_{\gamma}K$.
\end{lemma}

\begin{proof} Let $x \in G$ and $k,l \in K$ be such that $ k x^{-1}
\gamma x l = e$.  Then $x^{-1} \gamma x \in K$, which is not possible
by Lemma ~\ref{gamma}.
\end{proof}

\begin{proposition} \label{B}  There exists an open set $B$ in $G$
such that $C_{\gamma} \cap B$ is empty for all  $\gamma \in \Gamma_1
\cup \Gamma_2$ and $B$ is stable under $K\times K$ action on $G$ given
by equation \ref{biKaction}. 
\end{proposition}

\begin{proof}  Let $U'$ be a relatively compact open neighborhood of
$e$ in $G$.  Let $U = KU'K$. Then $U$ is relatively compact and hence
it intersects atmost finitely many conjugacy classes $C_{\gamma}$.
Since the map $G\to G//K$ is proper and the conjugacy class
$C_{\gamma}$ is closed, the set $KC_{\gamma}K$ is closed in $G$. Since
$U$ is $K$-stable, $KC_{\gamma}K \cap\ U$ is non-empty if and only if
$C_{\gamma}\cap U$ is non-empty.  Hence,  the set $E = \bigcup
\limits_{\gamma \neq e}[KC_{\gamma}K]\ \cap U$, being a finite union
of closed sets,  is a $K\times K$-stable closed subset of $U$. By
Lemma \ref{e}, the identity element $e$ does not belong to $E$. Choose
an open set $V \subseteq U$  containing $e$ such that  $E \cap V =
\emptyset$.  Let $B = KVK \cap K^{c}$, where $K^c$ is the complement
of $K$ in $G$.  It can be seen that B satisfies the desired  property.
\end{proof}

Now we give the proof of Theorem ~\ref{prin}.
\begin{proof} By hypothesis of Theorem ~\ref{prin}, there exists a
finite subset $S$ of $\widehat{G}_{s}$ such that 
\[m(\pi, \Gamma_1) = m(\pi, \Gamma_2)\ \ \forall\ \pi \notin S.\] Let
$f \in C^{\infty}_{c}(G//K)$. Since $f$ is bi-$K$-invariant,
$\chi_{\pi}(f) = 0$ if $\pi \notin \widehat{G}_{s} $.  Using the
Selberg trace formula for $f$ , we get :
\begin{equation}\label{trf5}
\begin{split} \sum \limits_{\pi\ \in\ S} \ t_{\pi} \chi_{\pi}(f) =
\sum \limits_{\substack{[\gamma]\ \in \ [\Gamma_{1}]_{G}\ \cup\
[\Gamma_{2}]_{G}}} (a(\gamma,\Gamma_1) - a(\gamma,\Gamma_2))\
O_{\gamma}(f) 
\end{split}
\end{equation} Let $\phi = \sum \limits_{\alpha\ \in\ S} t_{\pi}
\phi_{\pi}$. By using proposition ~\ref{warn}, we get :
\[ \int \limits_{G} f(g)\ \phi(g)\ d\mu(g) = \sum
\limits_{\substack{[\gamma]\ \in \ [\Gamma_{1}]_{G}\ \cup\
[\Gamma_{2}]_{G}  }} (a(\gamma,\Gamma_1) - a(\gamma,\Gamma_2))\
O_{\gamma}(f).\] Let $B$ be as in the proof of Proposition
~\ref{B}. The term on right hand side in above equation vanishes for
every function $f$ in $C^{\infty}_{c}(G//K)$ which is supported on
$B$. Hence for such functions $f$, 
  \[ \int \limits_{G} f(g)\ \phi(g)\ d\mu(g) = 0.\] By Lemma
\ref{sph}, the functions  $f$ separate points on $B$.  Hence  $\phi$
must vanish on the open subset $B$ of $G$.  Since  $\phi$ is analytic,
it  vanishes on all of G. By the linear independence of functions
$\phi_{\pi}$ (Proposition ~\ref{warn}), we conclude that
\[m(\pi, \Gamma_1) = m(\pi, \Gamma_2)\ \ \forall\ \pi \in
\widehat{G_s}.\]
\end{proof}

\section{\bf{Proof of Theorem ~\ref{laplacian}}}\label{prooflaplacian}
We now proceed to derive  Theorem ~\ref{laplacian} from Theorem
~\ref{prin}. We follow the notation given in the introduction. Let
$\pi$ be an irreducible, unitary,  spherical representation of
$G$. Let $e_{\pi}$ be a $K$-fixed vector of unit length in $\pi$.  The
associated spherical function $\phi_{\pi}$ is an eigenfunction of
$D(G/K)$ with eigencharacter $\lambda_{\pi}$:
\[ D(\phi_{\pi})=\lambda_{\pi}(D)\phi_{\pi} \quad D\in D(G/K).\] The
main observation is the following proposition. 

\begin{proposition}\label{pro1} Let $\Gamma$ be a torsion-free uniform
lattice in $G$.  Let $\pi$ be an irreducible, unitary  spherical
representation of $G$. Then 
\[m(\pi, \Gamma ) = \text{dim}\ (V(\lambda_{\pi},\Gamma)).\] In
particular,  $V (\lambda_{\pi},\Gamma)$ is finite dimensional.

 Conversely, if $\lambda$ is a character of $D(G/K)$ and the dimension
of $V (\lambda_{\pi},\Gamma)$ is positive, then $\lambda =
\lambda_{\pi}$ for some spherical representation $\pi$ of $G$.
\end{proposition}

\begin{remark}
 When $G =  PSL(2,\R)$ this is the duality theorem
proved by Gelfand, Graev and Pyatetskii-Shapiro  (\cite[p.50]{GGP}),
relating  the spectrum of the Laplace-Beltrami on the upper half plane
and the multiplicities of spherical representations of $PSL(2,\R)$
occurring in $L^2(\Gamma\backslash PSL(2, \R))$. We follow their
proof. The above fact is probably  well known to the experts but we
have included a proof for sake of completeness.  The proof also
indicates that Theorems \ref{prin} and Theorem \ref{laplacian} are
equivalent.
\end{remark}
 
\begin{proof}   Let $ \mathfrak{g}$ be the complexification of Lie
algebra of $G$ consisting of left invariant vector fields on $G$. Let
$\mathfrak{U}(\mathfrak{g})^{K}$ be the $K$-invariant subspace of the
universal enveloping algebra $\mathfrak{U}(\mathfrak{g})$ under the
right action of $K$ on $\mathfrak{U}(\mathfrak{g})$. We consider the
right action of $G$ on itself. This gives raise to a surjective map
from $\mathfrak{U}(\mathfrak{g})^{K}$ to $D(G/K)$ (\cite[page 52,
Proposition 1.7.5]{GV}).  Hence it can be seen that $D(e_{\pi})$  is a
$K$-fixed vector for each $D \in D(G/K)$. Since the dimension of the
space of $K$-fixed vectors of $\pi$ is one, it follows that $e_{\pi}$
is an eigenvector of  $D(G/K)$ with respect to the eigencharacter
$\lambda_{\pi}$ i.e. it lies in the eigenspace in $V(\lambda_{\pi},
\Gamma)$. Therefore, we conclude that $ m(\pi, \Gamma ) \leq
\text{dim}\ (V(\lambda_{\pi},\Gamma))$.

Conversely let $f \in C^{\infty}(X)$ be an eigenvector of some
character $\lambda$ of $D(G/K)$. Since 
\[L^{2}(\Gamma \backslash G) = \bigoplus_{\pi\ \in\ \widehat{G}} m
\left(\pi, \Gamma \right)\ \pi\] we write 
\begin{equation}\label{fdecomp}  f = \sum \limits_{\pi\ \in\
\widehat{G}} a_{\pi}\ v_{\pi}
\end{equation} such that $v_{\pi} \in \pi$ is a vector of unit
length. Let $W$ be the space of $K$-invariants of $L^{2}(\Gamma
\backslash G)$.  Let $P_W$ be the orthogonal projection of
$L^{2}(\Gamma \backslash G)$ onto $W$.  Since $f$ is right invariant
under $K$, $P_{W}(f) = f$. Hence we get :
\[ f = \sum \limits_{\pi\ \in\ \widehat{G}} a_{\pi}\ P_{W}(v_{\pi})\]

The algebra $D(G/K)$ is generated by essentially self-adjoint
differential operators. Hence, if the character $\lambda_{\pi}$ is
distinct from $\lambda$,  there exists an essentially self-adjoint
$D\in D(G/K)$ such that $\lambda_{\pi}(D)\neq \lambda(D)$.  Hence the
eigenvectors $v_{\pi}$ and $f$ are orthogonal. If $\pi$ is not a
spherical representation, $P_{W}(v_{\pi}) = 0$.  Hence the indexing
set in equation (\ref{fdecomp}) is restricted to those irreducible
unitary spherical representations with character $\lambda_{\pi}$ equal
to $\lambda$. 
  
Since the associated spherical functions to inequivalent
representations are linearly  independent,  the characters are
distinct. Hence we conclude that there is an unique irreducible
unitary spherical representation $\pi$ of $G$ such that $\lambda =
\lambda_{\pi}$. Hence,
\[ m(\pi, \Gamma ) = \text{dim}\ (V (\lambda_{\pi},\Gamma)).\]
\end{proof}

Now we give the proof of Theorem ~\ref{laplacian}.  Let $T$ be a
finite subset of characters of $D(G/K)$ such that \[\text{dim}\ ( V
(\lambda,\Gamma_1) ) = \text{dim}\ (V(\lambda,\Gamma_2))\] for all
characters $\lambda \notin T$. By above Proposition ~\ref{pro1}, we
get that :
\[ m(\pi,\Gamma_{1}) =  m(\pi, \Gamma_{2})\]  for all but finitely
many irreducible, unitary spherical representations of $G$. Hence
using Theorem ~\ref{prin} and Proposition ~\ref{pro1}, we get a proof
of Theorem \ref{laplacian}.

\end{document}